\documentclass[journal,twoside,web]{ieeecolor}

\usepackage{generic}
\usepackage{textcomp}
\usepackage{amsmath,amsfonts,amssymb,amscd}
\usepackage{enumerate}
\usepackage{mathrsfs}
\usepackage{mathtools}
\usepackage{graphicx}
\usepackage{subfigure}
\usepackage{listings}
\usepackage{hyperref}
\usepackage{titlesec}
\usepackage{float}
\usepackage{tabu}
\usepackage[svgnames]{xcolor}
\usepackage{tabularx}
\usepackage{listings}
\usepackage{multicol}
\usepackage{cite}
\usepackage{caption}
\usepackage[ruled,vlined]{algorithm2e}

\newtheorem{definition}{Definition}
\newtheorem{proposition}{Proposition}
\newtheorem{lemma}{Lemma}
\newtheorem{theorem}{Theorem}

\newtheorem{assumption}{Assumption}

\newcommand{\calX}{{\mathcal{X}}}

\newcommand{\calI}{{\mathcal{I}}}
\newcommand{\calB}{{\mathcal{B}}}
\newcommand{\calM}{{\mathcal{M}}}
\newcommand{\bbR}{{\mathbb{R}}}
\newcommand{\bbE}{{\mathbb{E}}}
\newcommand{\bbP}{{\mathbb{P}}}

\newcommand{\bx}{{\mathbf{x}}}

\newcommand{\by}{{\mathbf{y}}}
\newcommand{\bz}{{\mathbf{z}}}
\newcommand{\be}{{\mathbf{e}}}

\newcommand{\bG}{{\mathbf{G}}}
\newcommand{\bH}{{\mathbf{H}}}
\newcommand{\bJ}{{\mathbf{J}}}
\newcommand{\bw}{{\mathbf{w}}}
\newcommand{\bbN}{{\mathbb{N}}}
\newcommand{\me}{{\mathrm{e}}}
\newcommand{\utwi}[1]{\mbox{\boldmath $#1$}}

\newcommand{\bchi}{{\utwi{\chi}}}

\newcommand{\bvarepsilon}{{\utwi{\varepsilon}}}

\newcommand{\bxi}{{\utwi{\xi}}}

\title{Feedback-Based Optimization with Sub-Weibull Gradient Errors and Intermittent Updates}

\author{Ana M. Ospina, Nicola Bastianello, and Emiliano Dall'Anese
\thanks{A. Ospina and E. Dall'Anese are with Department of Electrical, Computer and Energy Engineering, University of Colorado Boulder, Boulder, CO, USA, \{ana.ospina, emiliano.dallanese\}@colorado.edu; N. Bastianello is with the Department of Information Engineering, University of Padova, Padova, Italy, {bastian4@dei.unipd.it}. This work was supported in part by the NSF CAREER award 1941896. }
}
\begin{document}

\allowdisplaybreaks[3]

\maketitle
\thispagestyle{empty}
\pagestyle{empty}

%%%%%%%%%%%%%%%%%%%%%%%%%%%%%%%%%%%%%%%%%
\begin{abstract}

This paper considers a feedback-based projected gradient method for optimizing systems modeled as algebraic maps. The focus is on a setup where the gradient is corrupted by random errors that follow a sub-Weibull distribution, and where the measurements of the output -- which replace the input-output map of the system in the algorithmic updates -- may not be available at each iteration.  The sub-Weibull error model is particularly well-suited in frameworks where the cost of the problem is learned via Gaussian Process (GP) regression (from functional evaluations) concurrently with the execution of the algorithm; however, it also naturally models setups where nonparametric methods and neural networks are utilized to estimate the cost.  Using the sub-Weibull model, and with Bernoulli random variables modeling missing measurements of the system output,  we show that the online algorithm generates points that are within a bounded error from the optimal solutions. In particular, we provide error bounds in expectation and in high probability. Numerical results are presented in the context of a demand response problem in smart power grids.

\end{abstract}


%%%%%%%%%%%%%%%%%%%%%%%%%%%%%%%%%%%%%%%%%%
\section{Introduction}
\label{sec:introduction} 

\IEEEPARstart{W}{e} consider optimization problems associated with systems  that feature  $M$ controllable inputs $\bx \in \mathbb{R}^M$ and unknown exogenous inputs $\bw \in \bbR^w$. Modeling the system as an algebraic map $\by= \calM(\bx, \bw)$, where $\calM: \bbR^M \times \bbR^w \to \bbR^y$ is well-defined~\cite{hauswirth2021optimization,bolognani2013distributed,Dallanese2019_PD}, the objective is to steer the system to optimal solutions of the following time-varying problem\footnote{\textit{\textbf{Notation}:} Upper-case (lower-case) boldface letters will be used for matrices (column vectors); $(\cdot)^\top$ denotes transposition. For a given vector $\bx \in \bbR^n$, $\| \bx \| := \sqrt{\bx^\top\bx}$. 
For a random variable $X \in \mathbb{R}$, $\bbE[X]$ denotes the expected value of $X$, and $\mathbb{P}[X \leq \epsilon]$ denotes the probability of $X$ taking values smaller than or equal to $\epsilon$. For a continuously-differentiable function $f: \bbR^n \to \bbR$, the gradient is denoted by $\nabla f: \bbR^n \to \bbR^n$. Given a convex set $\calX \subseteq \mathbb{R}^n$, $\mathrm{proj}_{\calX}(\bz)$ denotes the projection of $\bz \in \mathbb{R}^n$ onto $\calX$. $\mathcal{I}$ is the identity mapping. Finally, $e$ will denote the Euler's number.
}
\begin{equation} \label{eq:problem1}
    \bx_{*,t} \in \underset{\bx \in \calX_t}{\min} \; f_t(\bx) := C_t (\calM(\bx, \bw_t)) + U_t(\bx) ,
\end{equation}
where $t \in \mathbb{N} $ is the time index, $\calX_t \subseteq \bbR^M$ is a time-varying constraint set for the inputs, $\bx \mapsto U_t(\bx)$ is a cost associated with the inputs, and $\by \mapsto C_t(\by)$ is a cost associated with the outputs. We consider the case where $\bw_t$ is unknown or it cannot be directly measured; with $\bw_t$ \emph{not} known, \eqref{eq:problem1} can be solved using online algorithms of the following form (see, e.g.,\cite{bolognani2013distributed,Dallanese2019_PD}):
\begin{align} 
 \label{eq:update_ideal}
   \hspace{-.2cm} \bx_{t} = \mathrm{proj}_{\calX_t} \left[\bx_{t-1} - \alpha ( \bJ_t^\top \nabla C_t (\by_{t-1}) + \nabla U_t(\bx_{t-1})) \right],
\end{align}
where $\bJ_t := [\frac{\partial \mathcal{M}}{\partial x_m}(\bx_{t-1})]$ is the Jacobian  of  $\mathcal{M}$ ($x_m$ here  denotes the $m$th entry of $\bx$). The algorithm is ``feedback-based'' since  the output measurement $\by_{t-1}$ replaces the system model $\calM(\bx_{t-1}, \bw_{t-1})$ in the computation of the gradient. 

We consider an online algorithm similar to~\eqref{eq:update_ideal}, but in a setting where: (i) the gradient  is corrupted by random errors that follow a sub-Weibull distribution~\cite{Vladimirova_2020}; and, (ii)~measurements $\{\by_t\}_{t \in \mathbb{N} \cup \{0\}}$ are noisy and may not be available at each time $t$. The latter models processing and communication bottlenecks in the sensing layers of the system (for example, in power grid metering systems and transportation systems). On the other hand, the sub-Weibull model allows us to consider concurrent learning and optimization frameworks where the costs $\bx \mapsto U_t(\bx)$ and $\by \mapsto C_t(\by)$ are learned via Gaussian Process (GP) regression~\cite{GPML_Rasmussen06}, parametric methods~\cite{notarnicola2020distri}, non-parametric methods~\cite{hastie2009elements}, and neural networks~\cite{Marchi2022} from both a set of recorded data and (an infrequent set of) functional evaluations acquired during the execution of the algorithm. In this paper, we particularly highlight the use of  GPs~\cite{GP_userfeedback,fabiani2021learning}. Learning the cost concurrently with the execution of the algorithm finds ample applications in cyber-physical systems with human-in-the-loop~\cite{munir2013cyber}, where $U_t(\bx)$ models users' preferences, as well as data-enabled and perception-based optimization~\cite{lindemann2021learning} where $C_t(\by)$ is learned from data.
The ability of our framework to model various learning settings is grounded on the fact that the sub-Weibull distribution includes sub-Gaussian and sub-exponential errors as sub-cases, as well as random errors whose distribution has a finite support~\cite{Vladimirova_2020,Vershynin2018}.  

\emph{\textbf{Prior works}}. Prior works in the context of feedback-based optimization considered  time-invariant costs~\cite{bolognani2013distributed,chang2019saddle,hirata2014real} and time-varying costs~\cite{Dallanese2019_PD, colombino2019towards,tang2017real}; the convergence of algorithms were investigated when the cost is known, measurements are noiseless, and measurements are received at each iteration (see also the survey~\cite{hauswirth2021optimization} for a comprehensive list of references). 

Online optimization methods with concurrent learning of the cost with GPs were considered in \cite{GP_userfeedback} for functions satisfying the Polyak-\L ojasiewicz inequality; the algorithm employed the upper confidence bound and the regret was investigated. Quadratic functions were considered in \cite{notarnicola2020distri}, and they were estimated via recursive least squares. However, in~\cite{GP_userfeedback,notarnicola2020distri}, algorithms were not implemented with a system in the loop and no missing measurements were considered.  
Online algorithms with concurrent learning via shape-constrained GP were considered in~\cite{ospina2021time}; however, only bounds in expectation for the regret were provided. High probability convergence results were provided in~\cite{fabiani2021learning} for non-monotone games, where the pseudo-gradient is learned  from data. 

Although it is not the main focus of this paper, we also acknowledge works on zeroth-order methods (see, e.g.,~\cite{liu2018zeroth,tang2020distributed}) where gradient errors emerge from single- or multi-point gradient estimation; our framework based on a sub-Weibull model can  be applied to derive convergence bounds  when the gradient in~\eqref{eq:update_ideal} is estimated via single- or multi-point estimation. Finally, we mention that several convergence results have been derived for  classical online algorithms~\cite{selvaratnam2018Numerical,mokhtari2016online}, including asynchronous implementations~\cite{behrendt2021technical}; our results provide extensions to cases with sub-Weibull gradient errors. 

\emph{\textbf{Contributions}}. The main contributions of this paper are as follows.  \textit{C1)} We consider an online feedback-based projected gradient descent method with intermittent updates and with noisy gradients. We provide new  bounds for the error $\|\bx_t - \bx_{*,t}\|$ in \emph{expectation} that hold iteration-wise, where $\{\bx_t\}_{t \in \mathbb{N} }$ is the sequence generated by the algorithm; to this end, missing measurements are modeled as Bernoulli random variables. \textit{C2)}~We provide new bounds on $\|\bx_t - \bx_{*,t}\|$ in \emph{high probability}; the bounds are derived by adopting a sub-Weibull distribution for the error affecting the gradient.  \textit{C3)} We show how our framework models concurrent learning approaches, where one  leverages functional evaluations to learn the unknown cost via GPs during the execution of the algorithm. \textit{C4)} We test the performance of our algorithm in cases where the cost is estimated via GPs and feedforward neural networks.

The remainder of this paper is organized as follows. Section \ref{sec:pre} introduces preliminary definitions. Section \ref{sec:algo} presents the proposed algorithm and provides the corresponding analysis. Section \ref{sec:pre_GP} presents the concurrent learning approach. Section \ref{sec:sim} presents numerical results on a demand response problem, while Section \ref{sec:con} concludes the paper.

%%%%%%%%%%%%%%%%%%%%%%%%%%%%%%%%%%%%%%
% Problem statement 
%%%%%%%%%%%%%%%%%%%%%%%%%%%%%%%%%%%%%%
\section{Preliminaries} \label{sec:pre}

In this section, we introduce the class of sub-Weibull random variables and provide relevant properties. For a random variable  (rv) $X$, when the $k$-th moment of $X$ exists for some $k \geq 1$, we define $\|X \|_k := (\bbE [ |X |^k ] )^{1/k}$. 

\vspace{.1cm}

\begin{definition}[Sub-Weibull rv \cite{Vladimirova_2020}] \label{sec:def_subW} 
A random variable $X \in \mathbb{R}$ is sub-Weibull if  $\exists \, \theta > 0$ such that (s.t.) one of the following conditions is satisfied: 
\begin{enumerate}
	\item[(i)] $\exists \,\, \nu_1 > 0$ s.t. $\bbP[|X| \geq \epsilon] \leq 2 e^{- \left( \epsilon / \nu_1 \right)^{1 / \theta} }$, $\forall \, \epsilon~>~0$.
	\item[(ii)] $\exists \,\, \nu_2 > 0$ s.t. $\|X\|_k \leq \nu_2 k^\theta$, $\forall \, k \geq 1$. \hfill $\Box$
\end{enumerate}
\end{definition}
The parameters $\nu_1, \nu_2$ differ by a constant that depends on $\theta$; in particular, if property (ii) holds with parameter $\nu_2$, then property (i) holds with $\nu_1 = \left( 2 e / \theta \right)^\theta \nu_2$. Hereafter, we use the short-hand notation $X \sim \mathrm{subW}(\theta, \nu)$ to indicate that $X$ is a sub-Weibull rv according to Definition~\ref{sec:def_subW}(ii) (i.e.,  $\|X\|_k \leq \nu k^\theta$, $\forall \, k \geq 1$). We note that the sub-Weibull class includes sub-Gaussian and sub-exponential rvs as sub-cases; in particular, if $\theta = 1/2$ and $\theta = 1$ we have sub-Gaussian and sub-exponential rvs, respectively. Furthermore, if a rv has a distribution with finite support, it belongs to the sub-Gaussian class (by Hoeffding’s inequality \cite[Theorem 2.2.6]{Vershynin2018}) and, thus, to the sub-Weibull class.

\vspace{.1cm}

\begin{proposition} (\textit{Inclusion}~\cite{Vladimirova_2020}) 
Let $X \sim \mathrm{subW}(\theta, \nu)$ and let $\theta', \nu'$ s.t. $\theta' \geq \theta$, $\nu' \geq \nu$. Then, $X \sim \mathrm{subW}(\theta', \nu')$.  \hfill $\Box$ \label{sec:inclusion}
\end{proposition}

\vspace{.1cm}

\begin{proposition} (\textit{Closure of sub-Weibull class~\cite{bastianello2021stochastic}}) Let $X_i \sim \mathrm{subW}(\theta_i, \nu_i)$, $i = 1,2$, based on Definition~\ref{sec:def_subW}(ii). 
\begin{enumerate}	
	\item[(a)] \emph{Product by scalar:} Let $a \in \bbR$, then $a X_i \sim \mathrm{subW}(\theta_i, |a| \nu_i)$.
	\item[(b)] \emph{Sum by scalar:} Let $a \in \bbR$, then $a + X_i \sim \mathrm{subW}(\theta_i, |a| + \nu_i)$.
	\item[(c)] \emph{Sum:} Let $\{X_i, i = 1,2\}$ be possibly dependent; then, $X_1 + X_2 \sim \mathrm{subW}(\max\{ \theta_1, \theta_2 \}, \nu_1 + \nu_2)$.
	\item[(d)] \emph{Product:} Let $\{X_i, i = 1, 2\}$ be independent; then, $X_1 X_2 \sim \mathrm{subW}(\theta_1 + \theta_2, \nu_1 \nu_2)$. \hfill $\Box$
\end{enumerate} \label{sec:closure}
\end{proposition}

\vspace{.1cm}

\begin{proposition} (\textit{High probability bound \cite{Vladimirova_2020}})
Let $X \sim \mathrm{subW}(\theta, \nu)$ according to Definition ~\ref{sec:def_subW}(ii), for some $\theta > 0$ and $\nu > 0$. Then, for any $\delta \in (0,1)$, the bound:
\begin{equation}
\label{eq:highProbabilitybound}
| X | \leq \ \nu \log^\theta \left(\frac{2}{\delta} \right) \left( \frac{2e}{\theta} \right)^\theta
\end{equation}
\label{sec:h.p}
holds with probability $1-\delta$. \hfill $\Box$
\end{proposition}

%%%%%%%%%%%%%%%%%%%%%%%%%%%%%%
% Data-based Online Algorithm
%%%%%%%%%%%%%%%%%%%%%%%%%%%%%%
\section{Online Algorithm} \label{sec:algo}

In this section, we consider an online feedback-based algorithm of the form~\eqref{eq:update_ideal} to solve \eqref{eq:problem1} where: (i) we utilize noisy measurements of the output $\hat{\by}_t = \mathcal{M}(\bx_t, \bw_t) + \mathbf{n}_t$, where $\mathbf{n}_t \in \mathbb{R}^y$ is a measurement noise, instead of requiring full knowledge of the map $\mathcal{M}$ and of the vector $\bw_t$; (ii) we rely on \emph{inexact} gradient information; and, (iii) the measurements of $\by_t$ may not be received at each iteration. Accordingly, the online algorithm  is as follows (where we recall that $t \in \mathbb{N}$ is the time index):
\begin{align} \label{eq:algorithm}
\bx_{t} &=
    \begin{cases}
    \mathrm{proj}_{\calX_t} \left[\bx_{t-1} - \alpha (p_{t}(\hat{\by}_{t-1}) + s_t(\bx_{t-1})) \right] & \\
    & \hspace{-2.7cm} \text{if $\hat{\by}_{t-1}$ is received}\\
    \mathrm{proj}_{\calX_t} \left[\bx_{t-1} \right] &\hspace{-2.7cm} \text{if $\hat{\by}_{t-1}$ is not received,} 
    \end{cases} 
\end{align}
where $s_t(\bx) := \nabla U_t(\bx) + \bvarepsilon_t$ and $p_t({\hat{\by}}) := \bJ_t^\top {\nabla} C_{t}({\by}) + \bxi_t$  
are approximations of the gradients $\nabla U_t(\bx)$ and $\bJ_t^\top {\nabla} C_{t}({\mathcal{M}(\bx, \bw_t)}) $, respectively, with $\bvarepsilon_t \in \mathbb{R}^M$ and $\bxi_t \in \mathbb{R}^M$ random vectors that model the gradient errors. As discussed in Section~\ref{sec:introduction}, errors in the gradients emerge when a concurrent learning approach is utilized to estimate the functions $\bx \mapsto U_t(\bx)$ and $\by \mapsto C_t(\by)$ from samples using, e.g., GPs~\cite{GP_userfeedback},  parametric methods~\cite{notarnicola2020distri}, non-parametric methods~\cite{hastie2009elements}, or neural networks~\cite{Marchi2022}. Additional sources of errors include the noise affecting the measurement of $\by_t$. Hereafter, we denote the overall gradient error as $\be_t := \bvarepsilon_{t} + \bxi_{t}$. 

We rewrite~\eqref{eq:algorithm} in the following manner: 
\begin{align} 
    \bx_{t} &= \mathrm{proj}_{\calX_t} \left[\bx_{t-1} - v_{t-1} \, \alpha \left(p_{t}(\hat{\by}_{t-1}) + s_t(\bx_{t-1}) \right) \right],
    \label{eq:update_case2}
\end{align}
where $v_t$ is a rv taking values in the set $\{0,1\}$, and  is used to indicate whether the measurement of the system output is received or not. When a measurement is not received, we still utilize a projection onto the time-varying set $\calX_t$. In the following, we introduce the main assumptions and we analyze the performance of the online algorithm~\eqref{eq:update_case2}. 
% ----------------------------------------------
\subsection{Assumptions}
                               
Here, we outline the main assumptions used in the paper. 

\begin{assumption} \label{as:set_x} 
The set $\calX_t \subseteq \bbR^M$ is non-empty, convex and compact for all $t$.
\end{assumption}

\begin{assumption} \label{as:fun_f} 
For any $\bw \in \mathbb{R}^w$, the function $\bx \mapsto C_t(\mathcal{M}(\bx, \bw))$ is convex over $\bbR^M$, $\forall \, t$. Moreover, the composite function $\bx \mapsto f_t(\bx)$ is $\mu_t$-strongly convex and $L_t$-smooth over $\bbR^M$, for some  $0 < \mu_t \leq L_t < \infty$ and $\forall \, t$.  
\end{assumption}

We recall that $f_t(\bx)$ is $L_t$-smooth over $\bbR^M$, for some $L_t \geq 0$, if it is differentiable and  $ \|\nabla f_t(\bx) - \nabla f_t(\by) \| \leq L_t \| \bx - \by \|, \forall \, \bx, \by \in \bbR^M$. The previous assumptions imply that the norm of the gradient of $f_t$ is bounded over the compact set $\calX_t$. Assumption~\ref{as:fun_f} also implies that there is a unique optimizer $\bx_{*,t}$ for each $t$.  Furthermore, the map $\calI - \alpha \nabla f_t: \bbR^M \to  \bbR^M$ is $\zeta_t$-Lipschitz with $\zeta_t = \text{max}\{|1-\alpha \mu_t|, |1 - \alpha L_t| \}$ (see \cite[Section 5.1]{Ryu_2016});  if $\alpha \in \left(0, \frac{2}{L} \right), \, L = \sup_{1 \leq i \leq t} \{L_i\}$, then $\zeta_t < 1$; in other words, the map $\calI - \alpha \nabla f_t$ is contractive. 

\begin{assumption} \label{as:rv_v} 
The rv $v_t$ is Bernoulli distributed with parameter $p := \bbP [v_t = 1] > 0$. The rvs $\{v_t\}_{t \in \bbN \cup \{0\}}$ are i.i.d..
\end{assumption}

\begin{assumption} \label{as:error} 
For all $t \in \mathbb{N} \cup \{0\}$, $\exists~\theta_{\varepsilon} >0, \nu_{\varepsilon,{t}} >0$ s.t. each entry of the vector $\bvarepsilon_{t}$ is $\mathrm{subW} (\theta_{\varepsilon},\nu_{\varepsilon, {t}})$ according to Definition \ref{sec:def_subW}(ii). Moreover, $\bvarepsilon_t$ is independent of $v_t$ $\forall~t$.
\end{assumption}

\begin{assumption} \label{as:error2} 
For all $t \in \mathbb{N} \cup \{0\}$, $\exists~\theta_{\xi} > 0, \nu_{\xi, {t}}>0$ s.t. each entry of the vector $\bxi_{t}$ is $\mathrm{subW} (\theta_{\xi},\nu_{\xi,{t}})$, according to Definition \ref{sec:def_subW}(ii). Moreover,  $\bxi_t$ is independent of $v_t$ $\forall~t$.
\end{assumption}

The following lemma is then presented.  
\begin{lemma} \label{lm:error}
Suppose that Assumptions \ref{as:error}-\ref{as:error2} hold. Then, $\|\be_t\|$  is a sub-Weibull rv and, in particular, $\| \be_t \| \sim \mathrm{subW} (\max \{\theta_{\varepsilon}, \theta_{\xi} \}, (2^{\theta_{\varepsilon}} \sqrt{M} \nu_{\varepsilon, {t}}) + (2^{\theta_{\xi}} \sqrt{M} \nu_{\xi, {t}})), \, \forall~t$. \hfill $\Box$

\end{lemma}
This lemma can be proved by using~\cite[Lemma 3.4]{bastianello2021stochastic} and part (c) of Proposition~\ref{sec:closure}; the proof is omitted.

% ----------------------------------------------
\subsection{Convergence Analysis} \label{sec:performance}

Our analysis seeks bounds on the error  $\| \bx_t - \bx_{*,t} \|$, $t \in \mathbb{N}$,  where we recall that $\bx_{*,t}$ is the unique optimal solution of~\eqref{eq:problem1} at time $t$. To this end, we introduce the well-known definition of  path length  $\phi_{t}~:=~\|\bx_{*,t}-\bx_{*,t+1}\|$, which measures the temporal variability of the optimal solution of~\eqref{eq:problem1}. Finally, we let $E_t := \bbE [\|\be_t\|]$ to make the notation lighter. The main result of the paper is stated in the following.

\vspace{.1cm}

\begin{theorem} \label{th:1}
Let Assumptions \ref{as:set_x}-\ref{as:error2} hold, and let $\{\bx_t\}_{t \in \bbN} $ be a sequence generated by \eqref{eq:update_case2} for a given initial point $\bx_0 \in \mathcal{X}_0$. Recall that $\alpha \in (0, \frac{2}{L})$ and $\zeta_t = \text{max}\{|1-\alpha \mu_t|, |1 - \alpha L_t| \}$.
Then, the following holds for all $t \in \mathbb{N}$: 

\noindent \emph{\ref{th:1}(i)} The mean $\bbE \left[ \|\bx_{t} - \bx_{*, t}\| \right]$ is bounded as:  \vspace{-0.2cm} 
$$ \bbE \left[ \|\bx_{t} - \bx_{*, t}\| \right] \leq \beta_t \bbE [ \| \bx_0 - \bx_{*,0} \| ] + \sum_{i=1}^{t} [ \kappa_i \phi_{i-1} +  \alpha p \omega_i E_i ], $$ \vspace{-0.1cm} 
where, defining $\rho_t(\alpha) := 1 - p + p \zeta_t$, we have
\begin{equation*}
    \beta_{t} : = \prod_{i=1}^{t} \rho_i (\alpha), \quad
    \kappa_i = \begin{cases}
    \rho_i(\alpha) & \mathrm{if} \, i = {t}\\
    \prod_{k=i}^{t} \rho_k (\alpha) & \mathrm{if} \, i \neq t
    \end{cases},
\end{equation*} \vspace{-0.1cm} 
\begin{equation*}
    \omega_i = \begin{cases}
    1 & \mathrm{if} \, i = t\\
    \prod_{k=i+1}^{t} \rho_k (\alpha) & \mathrm{if} \, i \neq t
    \end{cases}.
\end{equation*}

\noindent \emph{\ref{th:1}(ii)} For any $\delta \in (0,1)$, the following bound holds with probability $1-\delta$: \vspace{-0.09cm} 
\begin{align}
    &\|\bx_{t} - \bx_{*,t}\| \leq \log^{\theta_x} \left(\frac{2}{\delta} \right) \left( \frac{2e}{\theta_x} \right)^{\theta_x} \bigg(  \eta(t) \| \bx_0 - \bx_{*,0} \| \nonumber \\
    & \hspace{2.5cm}  + \frac{1 - \zeta^t}{1 - \zeta} \sup_{0 \leq i \leq t} \left\{ \alpha \nu_{\me,i} + p^{-1} \phi_i  \right\} \bigg), \label{eq:highProb}
\end{align}
where $\zeta := \sup_{1 \leq i \leq t}{\zeta_i}$, $\theta_x := \max\{1, \theta_{\varepsilon}, \theta_{\xi} \} $, $\nu_{\me,t} :=  (2^{\theta_\varepsilon} \sqrt{M} \nu_{\varepsilon,{t}}) + (2^{\theta_\xi} \sqrt{M} \nu_{\xi,{t}})$, and where the function $t \mapsto \eta(t)$ is defined as $\eta(t)~:=~\max_k \left\{ \frac{\left( 1 - p + \zeta^k p \right)^{\frac{t}{k}} }{\sqrt{k}} \right\}$.  
\end{theorem}
\begin{proof}
See the Appendix.
\end{proof}

\vspace{0.1cm}

We note that, if $\alpha \in \left(0, \frac{2}{L} \right)$, then $\rho_t(\alpha) < 1$ for all $t$; in this case, 
letting $\rho(\alpha) := \sup_{1 \leq i \leq t} \{\rho_i(\alpha)\}$, the claim of Theorem \ref{th:1}(i) implies that 
\begin{align}
     \bbE \left[ \|\bx_{t} - \bx_{*, t}\| \right] & \leq \rho(\alpha)^t \bbE [ \| \bx_0 - \bx_{*,0} \| ] \nonumber \\
     & \hspace{-1.6cm} + \frac{1}{1 - \rho(\alpha)} \sup_{0 \leq i \leq t}\{ \phi_i \}  +  \frac{\alpha  p}{1 - \rho(\alpha)} \sup_{0 \leq i \leq t} \{E_i \}  \label{eq:bound_exp}
\end{align}
where the term $\rho(\alpha)^t \bbE [ \| \bx_0 - \bx_{*,0} \| ] \rightarrow 0$ as $t \rightarrow \infty$. From~\eqref{eq:bound_exp}, it is evident that the error $\|\bx_{t} - \bx_{*, t}\|$  is asymptotically bounded in expectation, with an error bound that depends on the variability of the optimal solution, the mean of the norm of the gradient error, and the Bernoulli parameter $p$.  
We also draw a link with stability of stochastic discrete-time systems~\cite{teel2013equivalent,jiang2001input} by noting that~\eqref{eq:bound_exp} establishes that the stochastic algorithm~\eqref{eq:update_case2} renders the set $\{0\}$ exponentially input-to-state stable (E-ISS) in expectation. 

Theorem \ref{th:1}(ii) asserts that $\|\bx_{t} - \bx_{*, t}\|$ is bounded in high probability. Since $t \mapsto \eta(t)$ is monotonically decreasing,  $\eta(t) \rightarrow 0$ as $t~\rightarrow~\infty$; thus, Theorem \ref{th:1}(ii) provides an asymptotic error bound that holds in high probability. We also note that the first term on the right-hand-side of~\eqref{eq:highProb} is a $\mathcal{KL}$ function~\cite{jiang2001input}; upper-bounding the second term as $\frac{1}{1 - \zeta} \sup_{0 \leq i \leq t} \left\{ \alpha \nu_{\me,i} + p^{-1} \phi_i \right\}$, we obtain a result in terms of  ISS in high-probability  under a sub-Weibull error model.

%%%%%%%%%%%%%%%%%%%%%%%%%%%%%%%%%%%%%%%%
% OPTIMIZATION WITH GP
%%%%%%%%%%%%%%%%%%%%%%%%%%%%%%%%%%%%%%%%
\section{Optimization with Concurrent GP Learning} \label{sec:pre_GP} 

In this section, we provide an example of a framework where the cost function is estimated using GPs during the execution of the online algorithm~\eqref{eq:update_case2}. We will show that the results of Theorem \ref{th:1} directly apply to this framework. 

To streamline exposition, suppose that the function $\by \mapsto C_t(\by)$ is known, and the function $\bx \mapsto U(\bx)$ is static but unknown (a similar approach can be used for time-varying functions). Furthermore, suppose that $U(\bx) = \sum_{m=1}^M u_{m}(x_m)$ where $u_m: \mathbb{R} \rightarrow \mathbb{R}$ is a cost associated with the $m$-th input. We then consider a concurrent learning approach where we estimate each function $x_m \mapsto u_{m}(x_m)$ via GP, based on noisy functional evaluations~\cite{GP_userfeedback}.

A GP is a stochastic process and is specified by its mean function and its covariance function~\cite{GPML_Rasmussen06}. Accordingly, let $u_{m}(x)$ be characterized by a GP, i.e., for any $x, x' \in \mathcal{X} \subseteq \mathbb{R}$, $\mu_{m}({x}) = \mathbb{E}[u_{m}({x})]$ and $k_{m}({x}, {x}') = \mathbb{E}[(u_{m}({x}) - \mu_{m}({x}))(u_{m}({x}') - \mu_{m}({x}'))]$. Let $\bchi_{m,t} = [ x_{m,t_1}  \in \mathcal{X}, \dots, x_{m,t_q}  \in \mathcal{X}]^\top$ be the set of $q$ sampling points at times $\{t_i\}_{i = 1}^q \subset \{0, \ldots, t\}$; let ${z}_{m,t_i}  = u({x}_{m,t_i}) + \varpi_{t_i}$, with $\varpi_{t_i} \stackrel{\mathrm{iid}}{\sim} \mathcal{N}({0}, \sigma^2)$ Gaussian noise, be the noisy functional evaluation at $x_{m,t_i}$; finally, define $\mathbf{z}_{m,t} = [z_{m,t_1}, \dots, z_{m,t_q}]^\top$. Then, the posterior distribution of $(u_{m}(x)|\bchi_{m,t}, \bz_{m,t})$ is a GP  with mean $\mu_{m,t}({x})$, covariance $k_{m,t}({x}, {x}')$, and variance $\varsigma^2_{m,t}(x)$ given by~\cite{GPML_Rasmussen06}:
\begin{subequations} \label{eq:GP_update}
\begin{align}
    \mu_{m,t}({x}) &= \mathbf{k}_{m,t}(x)^\top (\mathbf{K}_{m,t} + \sigma^2 \mathbf{I})^{-1} \mathbf{z}_{m,t}, \label{eq:GP_update_mu} \\
    k_{m,t}({x}, {x}') &= k_m(x,x') - \mathbf{k}_{m,t}(x)^\top (\mathbf{K}_{m,t} + \sigma^2 \mathbf{I})^{-1} \mathbf{k}_{m,t}(x'), \nonumber  \\
    \varsigma_{m,t}^2(x) &= k_m(x,x), \label{eq:GP_update_si}
\end{align}
\end{subequations}
where $\mathbf{k}_{m,t}(x) = [k_m(x_{m,t_1},x), \dots, k_m(x_{m,t_q},x)]^\top$, and $\mathbf{K}_{m,t}$ is the positive definite kernel matrix $[k_m(x,x')]$. For example, using a squared exponential kernel, $k_m({x},{x}')$ is given by $k_m({x},{x}')~=~\sigma_{f}^2 e^{-\frac{1}{2\ell^2}({x}-{x}')^{2}}$, where the hyperparameters are the variance $\sigma_{f}^2$ and the characteristic length-scale $\ell$~\cite{GPML_Rasmussen06}. 

The idea is then to utilize the posterior mean $\mu_{m,t}(x)$, computed via~\eqref{eq:GP_update_mu} based on the samples collected up to the current time $t$, as an estimate of the function $u_m(x)$. Accordingly, the function $U(\bx)$ can be approximated at time $t$ as $ \mu_{t}(\bx) = \sum_{m = 1}^M \mu_{m,t}(x_m)$ and $s_t(\bx)$ in~\eqref{eq:update_case2} can be set to $s_t(\bx) = \nabla \mu_{t}(\bx)$. The resulting GP-based learning framework would involve the sequential execution of the online algorithm~\eqref{eq:update_case2}, with $s_t(\bx) = \nabla \mu_{t}(\bx)$, and where the estimates $\{\mu_{m,t}(x)\}$ are updated via~\eqref{eq:GP_update_mu} whenever a new functional evaluation becomes available. In this case, the $m$-th entry of the gradient error vector $\bvarepsilon_{t}$  can be expressed as $\varepsilon_{m,t}(x_m) = \frac{d u_m}{d x_m}(x_m) - \frac{d \mu_{m,t}}{d x_m}(x_m)$.  

Since the function $u_m(x_m)$ is modeled as a GP, its derivative is also a GP~\cite{GPML_Rasmussen06}. For a given $x_m \in \mathbb{R}$, it follows that the error $\varepsilon_{m,t}(x_m)$ is a Gaussian random variable~\cite{GPML_Rasmussen06} (and, hence, sub-Gaussian~\cite{Vershynin2018}).  Since the  class of sub-Weibull rvs includes sub-Gaussian distributions by simply setting $\theta = 1/2$~\cite{Vladimirova_2020}, it follows that $\varepsilon_{m,t}(x_m) \sim \mathrm{subW}(1/2, \nu_{\varepsilon,t})$, for some $\nu_{\varepsilon,t} > 0$. 

Summarizing, when GPs are utilized to estimate the function $\bx \mapsto U(\bx)$, Assumption~\ref{as:error} is satisfied (similar arguments hold if we utilize GPs to estimate the function $\by \mapsto C_t(\by)$). In particular, it  holds that 
$\|\bvarepsilon_t\| \sim \mathrm{subW}(\theta_{\varepsilon},  2^{\theta_{\varepsilon,{t}}}\sqrt{M} \nu_{\varepsilon,{t}})$, with $\theta_{\varepsilon} = 1/2$.

%%%%%%%%%%%%%%%%%%%%%%%%%%%%%%%%%%%
% SIMULATIONS
%%%%%%%%%%%%%%%%%%%%%%%%%%%%%%%%%%%
\section{Numerical Results} \label{sec:sim}

We consider an application in the context of demand response in power distribution systems \cite{Lesage_2020}. Here, $\bx$ is the vector of active power setpoints from controllable distributed energy resources (DERs), $\bw_t$ is the vector of powers consumed by non-controllable loads, $\by_t$ represents the net real power exchanged at some points of common coupling (PCC), and the map $\calM(\bx, \bw_t) =  \bG \bx + \bH \bw_t$ is built based on a linearization of the power flow equations~\cite{bolognani2013distributed}. We assume that the powers consumed by non-controllable loads cannot be individually measured; rather, measurements of $\by_t$ are available from meters and sensing units. The function $U_t(\bx)$ represents the dissatisfaction of the users (e.g., relative to indoor temperature if the DER is an AC unit, or charging rate of an electric vehicle); finally, we consider the function $C_t(\by_t) = \frac{\beta}{2} \|\by_t - \by_{t,\mathrm{ref}} \|^2$, with $\by_{t,\mathrm{ref}}$ a time-varying demand response setpoint for the PCCs, and $\beta > 0$ a given parameter.

\begin{figure}[t!]
  \centering
  \includegraphics[width=0.48\textwidth]{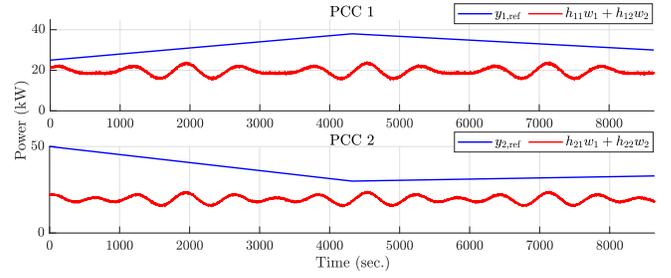}
  \caption{\small Reference signal $\by_{\mathrm{ref}}$ and overall contribution of the non-controllable loads at the point of common coupling. }
  \label{fig:ref}
\end{figure}
\begin{figure}[t!]
  \centering
  \begin{subfigure}[]{\includegraphics[width=0.48\textwidth]{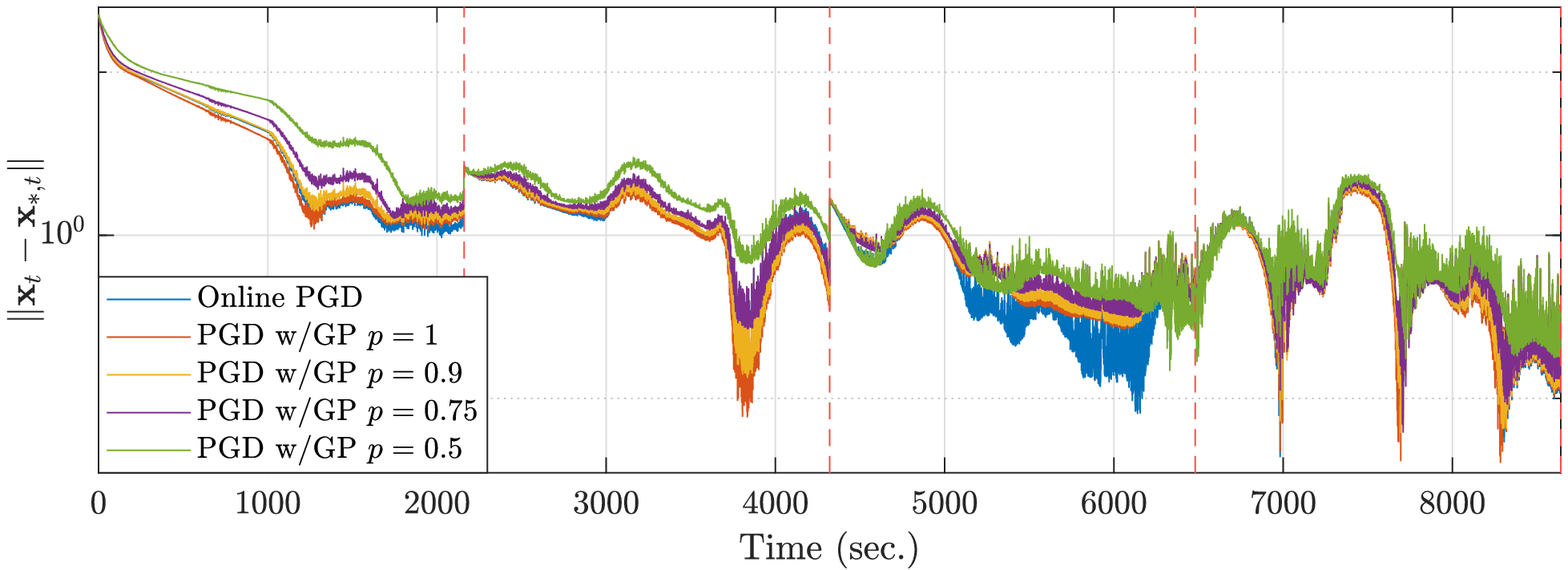}}
  \end{subfigure}
  \begin{subfigure}[]{\includegraphics[width=0.48\textwidth]{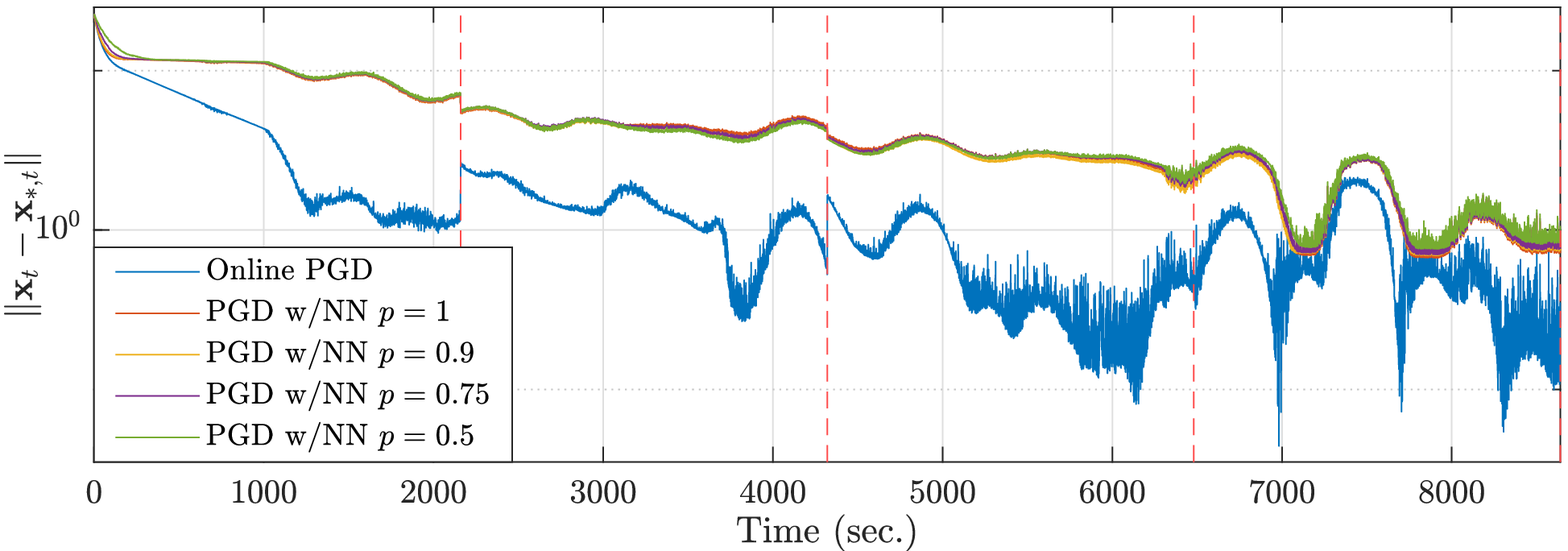}}
  \vspace{-0.2cm}
  \end{subfigure}
  \caption{ \small Error $\|\bx_t - \bx_{*,t}\|$ for different values of $p$. (a) Learning with GPs. (b) Learning with feedforward NNs. In both cases, the blue line corresponds to the online algorithm with exact knowledge of $\{U_{m}\}$. }
  \vspace{-0.4cm}
  \label{fig:tracking_p}
\end{figure}

We consider $M = 6$  controllable DERs (accordingly, $x_m \in \mathbb{R}$ is the active power produced or consumed by the DER $m$), and $2$ PCCs. The power limits for the DERs are time-variant and change within the following bounds: $\calX_{1,t} = [\{-10,-6\}; \{6,10\}]$kW, $\calX_{2,t} = [\{3,7\}; \{13,17\}]$kW, and $\calX_{3,t}= [\{0,3\}; \{28,32\}]$kW. The aggregate power of the non-controllable loads are shown in red in Figure \ref{fig:ref}, and the two demand response setpoints are color-coded in blue. The function $U_t(\bx)$ is unknown; it is assumed that $U_t(\bx)$ switches between two quadratic functions (with different coefficients), to reflect changes in the preferences of the DER-owners (switching times are represented by vertical red lines in Figure \ref{fig:tracking_p}); for example, a DER owner can change preferences for the indoor temperature. We test two learning methods: (i) $U_t(\bx)$ is estimated via GPs as explained in Section~\ref{sec:pre_GP}, and (ii) we use feedforward neural networks (NNs) to estimate the functions associated with the DERs. We evaluate the performance of the algorithm in \eqref{eq:update_case2} over a period equivalent to 12 hours, with each step of the algorithm performed every 5 seconds. We start to learn each of the users' function with 5 noisy observations, and we collect functional evaluations from the DER owners every 30 minutes during the execution of the algorithm. 

Figure \ref{fig:tracking_p}(a) illustrates the performance of the online method when GPs are utilized (we use the labels ``PGD w/GP''), averaged  over $10$ experiments. For each experiment, we use a fixed step size $\alpha = 0.5$ and a  random initial point $\bx_{0}$. We show the mean error for four values of the Bernoulli parameter $p$; for comparison purposes, we run the online proximal gradient method with exact knowledge of $U_t(\bx)$ (labeled as ``online PGD''). The tracking error decreases linearly and then settles to a range of values that depend on $p$ as expected. It can also be seen that, when $p = 1$, the red trajectory and the blue one are very close after $6000$ steps, indicating that the GP approximates well the function $U_t(\bx)$. Figure \ref{fig:tracking_p}(b) presents the case when $U_t(\bx)$ is learned via a feedforward NNs with one hidden layer of size 10 (labeled as ``PGD w/NN''). The online algorithm exhibits a similar behavior; however, the tracking error is in general higher compared to the case where we use GPs. One reason for this behavior is the higher error in the gradient estimation; while GPs offer a closed-form expression for the gradient of the posterior mean, in the case of the NNs we estimated the gradient via centered difference. 

%%%%%%%%%%%%%%%%%%%%%%%%%%%%%%%%%%%%%%%%
% CONCLUSIONS
%%%%%%%%%%%%%%%%%%%%%%%%%%%%%%%%%%%%%%%%
\section{Conclusions} \label{sec:con}

We considered a feedback-based projected gradient method   to solve a time-varying optimization problem associated with a  system modeled with an algebraic map. The algorithm relies on inaccurate gradient information and exhibits random updates.  We derived bounds for the error between the iterate of the algorithm and the optimal solution of the optimization problem in expectation and in high probability, by modeling gradient errors as sub-Weibull rvs and missing measurements as Bernoulli rvs. We established a connection with results in the context of ISS in expectation and in high probability for discrete-time stochastic dynamical systems. 

% REFERENCES
\bibliographystyle{IEEEtran}
\bibliography{References}

%%%%%%%%%%%%%%%%%%%%%%%%%%%%%%%%%%%%%%%%
% PROOF
%%%%%%%%%%%%%%%%%%%%%%%%%%%%%%%%%%%%%%%%
\appendix

\textbf{Proof of Theorem 1}. The proof of the theorem utilizes the definition of sub-Weibull rv in Definition \ref{sec:def_subW}(ii). To derive the main result, it is first necessary to characterize the rv $\zeta^{\Omega_t}$, where $\Omega_t := \sum_{i = 0}^{t - 1} v_i$ and $\zeta := \sup_{1 \leq i \leq t}{\zeta_i}$, $\zeta \in (0,1)$. We will find the parameters $\theta > 0$ and $\eta(t) > 0$ s.t. $\zeta^{\Omega_t}$ can be modeled as $\zeta^{\Omega_t} \sim \mathrm{subW}(\theta, \eta(t))$. By definition, $\Omega_t$ is a binomial rv, i.e., $\Omega_t \sim \calB(p,t)$ since it is the sum of $t$ Bernoulli trials. Further, $\zeta^{\Omega_t} \in [\zeta^{t}, 1)$, which implies that $\zeta^{\Omega_t}$ is a bounded rv. Hence, we can model $\zeta^{\Omega_t}$ as a sub-Gaussian rv~\cite{Vershynin2018} and, thus, a sub-Weibull with $\theta = 1/2$.

Regarding $\eta(t)$, by the definition of the $k$-th moment of a rv, we have that the $k$-th moment of the bounded rv $\zeta^{\Omega_t}$ is \vspace{-0.05cm}
\begin{align*}
    \left\| \zeta^{\Omega_t} \right\|_k^k &= \bbE \left[ \left( \zeta^{\Omega_t} \right)^k \right] = \bbE \left[ \left( \zeta^k \right)^{\Omega_t} \right] \\
    &\stackrel{\mathrm{(a)}}{=} \sum_{h=0}^{t} (\zeta^k)^{h}  {t \choose h} p^h (1-p)^{t-h}\\
    &\stackrel{\mathrm{(b)}}{=} \sum_{h=0}^{t}  {t \choose h} (\zeta^k p)^{h}  (1-p)^{t-h} = (1-p +\zeta^k p )^{t}, 
\end{align*}
where (a) follows by the definition of expected value and probability mass function of the binomial rv; and, (b) uses the binomial identity. Thus, \vspace{-0.05cm}
\begin{equation} \label{eq:k_bound}
    \left\| \zeta^{\Omega_t} \right\|_k = (1-p +\zeta^k p )^{\frac{t}{k}}. 
\end{equation} 
We can see from \eqref{eq:k_bound} that the $k$-th moment of $\zeta^{\Omega_t}$, for a fixed $k$, decays to zero as $t \to \infty$. On the other hand, if we fix a finite $t \in \bbN$, we have that $ \left\| \zeta^{\Omega_t} \right\|_k \to 1$ as $k \to \infty$. By Definition \ref{sec:def_subW}(ii) of the sub-Weibull rv, we have that
$$ \eta(t) \geq \frac{\|\zeta^{\Omega_t}\|_k}{\sqrt{k}} =  \frac{(1-p +\zeta^k p )^{\frac{t}{k}}}{\sqrt{k}}, \qquad \forall k \geq 1.$$ Therefore, $\eta(t)$ is a decreasing function of $t$, and it takes values in the set $(0, 1)$; also, $\eta(t) \rightarrow 0^+$ for $k~\to~\infty$. Thus, for any given $t$ and any finite $k$, we can choose $\eta(t)~=~\max_k \left\{ \frac{\left( 1 - p + \zeta^k p \right)^{\frac{t}{k}} }{\sqrt{k}} \right\}~\in~ (0, 1)$.  

With this characterization in place, we now derive a bound for $d_{t} := \|\bx_{t} - \bx_{*, t}\|$. Notice that $\bx_{*, t}$ satisfies the fixed-point equation $\bx_{*,t}~:=~\mathrm{proj}_{\calX_{t}} \left[ \bx_{*,t} - \alpha \nabla f_t(\bx_{*,t})\right]$. Then, 
% ------------
\small \begin{align} \nonumber
    d_{t+1} &\stackrel{\mathrm{(a)}}{=} \| v_t \, \mathrm{proj}_{\calX_{t+1}} \left[\bx_t - \alpha \left( \nabla f_{t+1}(\bx_t) + \be_t \right) \right] - \bx_{*, t+1} \\ \nonumber
    &\quad + (1 - v_t) \, \mathrm{proj}_{\calX_{t+1}} \left[ \bx_t \right] \| \\ \nonumber
    &\stackrel{\mathrm{(b)}}{=} \| \mathrm{proj}_{\calX_{t+1}} \left[ \bx_t \right] + v_t \, \mathrm{proj}_{\calX_{t+1}} \left[\bx_t - \alpha \left( \nabla f_{t+1}(\bx_t) + \be_t \right) \right] \\ \nonumber
    &\quad - v_t \mathrm{proj}_{\calX_{t+1}} \left[ \bx_t \right] - \bx_{*, t+1} + v_t \bx_{*, t+1} - v_t \bx_{*, t+1} \| \\ \nonumber
    &\stackrel{\mathrm{(c)}}{\leq} \| \mathrm{proj}_{\calX_{t+1}} \left[ \bx_t \right] - \bx_{*, t+1} + v_t \bx_{*, t+1} - v_t \mathrm{proj}_{\calX_{t+1}} \left[ \bx_t \right] \| \\ \nonumber
    &\quad + v_t  \| \left[\bx_t - \alpha \left( \nabla f_{t+1}(\bx_t) + \be_t \right) \right] -  \bx_{*, t+1} \|\\ \nonumber
    &\stackrel{\mathrm{(d)}}{\leq}(1-v_t)\| \bx_t - \bx_{*, t+1} \|  + v_t \zeta_{t+1} \| \bx_t - \bx_{*, t+1} \| + v_t \alpha \|  \be_t \| \\ 
    &\stackrel{\mathrm{(e)}}{\leq} \zeta_{t+1}^{v_t} d_t + \zeta_{t+1}^{v_t}  \phi_{t} + v_t \alpha \|  \be_t \|.
    \label{eq:case2_P1}
\end{align} \normalsize
where (a) holds by \eqref{eq:update_case2}; (b) by adding and subtracting $\bx_{*,t+1}$; (c) by reorganizing terms, using the triangle inequality, and the non-expansiveness property of the projection operator into the second term; (d) by using the triangle inequality on the second term, using the fact that $\calI - \alpha \nabla f_{t+1}$ is $\zeta_{t+1}$-Lipschitz, and the non-expansiveness property of the projection; and, (e) holds by adding and subtracting $\bx_{*,t}$, using the triangle inequality, by the definition of $\phi_t$, and $\zeta_{t+1}^{v_t} := 1-v_t + v_t \zeta_{t+1}$.

To show \emph{1(i)}, take the expectation of \eqref{eq:case2_P1} to obtain
\begin{align} \nonumber
    \bbE \left[ d_{t+1} \right] &\leq \bbE [ \zeta_{t+1}^{v_t} d_t  + \zeta_{t+1}^{v_t} \phi_t + v_t \alpha \|\be_t\| ]\\ \nonumber
    &\stackrel{\mathrm{(a)}}{=} \bbE [ \zeta_{t+1}^{v_t} d_t ] + \bbE [ \zeta_{t+1}^{v_t} \phi_t ] + \bbE [ v_t \alpha \|\be_t \| ] \\ \nonumber
    &\stackrel{\mathrm{(b)}}{=} \bbE[\zeta_{t+1}^{v_t} ] \bbE [d_t] + \bbE[\zeta_{t+1}^{v_t} ] \phi_t + \alpha \bbE [ v_t ] \bbE [ \|\be_t\| ] \\
    &\stackrel{\mathrm{(c)}}{\leq} \rho_{t+1}(\alpha) \bbE [ d_t ] + \rho_{t+1}(\alpha) \phi_t + \alpha p E_t,
    \label{eq:case2_P2}
\end{align}
where (a) holds by the linearity of the expected value; (b) by the independence of the rvs $v_t$ and $\|\be_t\|$; and, (c) holds by the expected value of the Bernoulli rv $v_t$, where we also used $\rho_t(\alpha) \geq \zeta_t$. Recursively applying \eqref{eq:case2_P2}, we have,
\begin{align} \nonumber
    \bbE \left[ d_t \right] &\leq \beta_t \bbE [ d_0 ] + \sum_{i=1}^{t} \kappa_i \phi_{i-1} + \alpha p  \sum_{i=1}^{t}\omega_i E_i.
    \label{eq:case2_P3}
\end{align}

To show \emph{1(ii)}, let $\zeta := \sup_{1 \leq i \leq t} \{\zeta_i\} \in (0,1)$; obviously, $\zeta_t \leq \zeta \, \forall t$. Then
\begin{equation}
    d_{t+1} \leq \zeta^{v_t} d_t + \zeta^{v_t}  \phi_t + v_t \alpha \|\be_t \|,
    \label{eq:case2_P4}
\end{equation}
almost surely, and iterating \eqref{eq:case2_P4} we have
\begin{align} \nonumber
   & \hspace{-.4cm} d_t \leq \prod_{i=1}^{t} \zeta^{v_i} d_0 + \sum_{i=1}^{t} \prod_{k=i}^{t} \zeta^{v_k} \phi_{i-1} + \alpha \sum_{i=1}^{t} \prod_{k=i+1}^{t} \zeta^{v_k} v_i \|  \be_i \| \\
    & \hspace{-.4cm} =  \zeta^{\Omega_t} d_0 + \sum_{i=1}^{t} \zeta^{\left( \sum_{k=i}^{t} v_k \right)} \phi_{i-1} + \alpha \sum_{i=1}^{t} \zeta^{\left( \sum_{k=i+1}^{t} v_k \right)} v_i \|  \be_i \|. \hspace{-.2cm}
    \label{eq:case2_P5}
\end{align}

Define the sub-sequence $\{i_j\}_{j=1}^{\Omega_t}$ with $v_{i_j} = 1$ for $j~=~1, \dots, \Omega_t$; i.e., $\{j\}$ are the indices of the iterations where an update is performed. Then, 
\begin{equation*}
    \sum_{i=1}^{t} \zeta^{\left( \sum_{k=i+1}^{t} v_k \right)} v_i \| \be_i \| = \sum_{j=1}^{\Omega_t} \zeta^{\left( \sum_{k=i_j+1}^{t} v_k \right)} \|  \be_{i_j} \|.
\end{equation*}

By definition of $\{i_j\}_{j=1}^{\Omega_t}$, between the times $i_j + 1$ and $t$ the total number of updates is $\Omega_t - j$; then, we rewrite \eqref{eq:case2_P5} as follows: 
\begin{align} \nonumber
    d_{t} &\leq \zeta^{\Omega_t} d_0 + \alpha \sum_{j=1}^{\Omega_t} \zeta^{\Omega_t - j} \|  \be_{i_j} \| + \sum_{i=1}^{t} \zeta^{\left( \sum_{k=i}^{t} v_k \right)} \phi_{i-1} \\
    &\leq \zeta^{\Omega_t} d_0 + \alpha \sum_{j=1}^{t} \zeta^{j} \|  \be_{t - i_j} \| + \sum_{i=1}^{t} \zeta^{\left( \sum_{k=i}^{t} v_k \right)} \phi_{i-1},
    \label{eq:case2_P6}
\end{align}
where $t - \Omega_t$ terms are added to the sum to remove the dependence on $\Omega_t$. Recall that $\zeta^{\Omega_t} \sim \mathrm{subW}(1/2, \eta(t))$, where $t \mapsto \eta(t)$ is monotonically decreasing. By Assumptions~\ref{as:error}-\ref{as:error2} and Lemma \ref{lm:error}, we have that $\| \be_t \| \sim \mathrm{subW} (\theta_{\me},\nu_{\me,t}) \, \forall t$,  
where $\theta_{\me} = \max \{\theta_{\varepsilon}, \theta_{\xi} \}$ and $\nu_{\me, t} =  (2^{\theta_\varepsilon} \sqrt{M} \nu_{\varepsilon, {t}}) + (2^{\theta_\xi} \sqrt{M} \nu_{\xi, {t}})$. Then, using Proposition~\ref{sec:closure}, we get:  
$$\zeta^{\Omega_t} d_0 + \alpha \sum_{j=1}^{t} \zeta^{j} \|  \be_{t - i_j} \| \sim \mathrm{subW}(\theta', \nu'), \vspace{-0.1cm} $$  where $ \theta' = \max\{1/2, \max \{\theta_{\varepsilon}, \theta_{\xi} \} \}$, 

    $\nu' = \eta(t) d_0 + \alpha \frac{1 - \zeta^t}{1 - \zeta} \sup_{0 \leq i \leq t} \{ (2^{\theta_\varepsilon} \sqrt{M} \nu_{\varepsilon,{i}}) + (2^{\theta_\xi} \sqrt{M} \nu_{\xi,{i}}) \}$,

and where we used the closure of the sub-Weibull rv with respect to sum and product (with a scalar), and then the inclusion property.

To characterize the last term of \eqref{eq:case2_P6} we have that
\begin{align*} 
    \sum_{i=1}^{t} \zeta^{\left( \sum_{k=i}^{t} v_k \right)} \phi_{i-1} = \sum_{j=1}^{\Omega_t} \zeta^j \sum_{i = i_j}^{i_{j+1}} \phi_{i} \leq \sum_{j=1}^{t} \zeta^j \sum_{i = i_j}^{i_{j+1}} \phi_i.
\end{align*} 
Note that $\sum_{i = i_j}^{i_{j+1}} \phi_{i}$ is the sum of a given number of the deterministic path lengths $\phi_t$; then we get
\begin{align} \label{eq:phi}
    \hspace{-0.25cm} \sum_{i = i_j}^{i_{j+1}} \phi_{i}    &{\leq} \sup_{0 \leq i \leq t} \{\phi_{i}\}  \sum_{i = i_j}^{i_{j+1}} 1 =  \sup_{0 \leq i \leq t} \{ \phi_{i} \} (i_{j+1}-i_j).
\end{align}
The geometric rv $i_{j+1} - i_j $ can be characterized as a sub-Weibull rv. Since the exponential distribution is the continuous analogue of geometric distribution, and by logarithmic properties, we have $(i_{j+1}- i_j)~\leq~Z \sim \mathrm{subW}(1, 1/p)$; by Proposition \ref{sec:closure}, we have that \eqref{eq:phi} $\sim \mathrm{subW}(1,\sup_{0 \leq i \leq t} \{\phi_i\} \, p^{-1})$. Therefore, \eqref{eq:case2_P6} follows a sub-Weibull distribution with parameters $\max\{1, \theta_{\me}\}$ and $\eta(t) d_0 + \frac{1 - \zeta^t}{1 - \zeta} \sup_{0 \leq i \leq t} \left\{ \alpha \nu_{\me,i} + p^{-1} \phi_{i-1} \right\}$. Using the high probability bound~\eqref{eq:highProbabilitybound} in Proposition~\ref{sec:h.p} the result follows.

\end{document}